\documentclass[10pt]{amsart}
\usepackage{color}
\newtheorem{theorem}{Theorem}[section]
\newtheorem{lemma}[theorem]{Lemma}
\newtheorem{q}[theorem]{Question}
\newtheorem{prop}[theorem]{Proposition}
\newtheorem{corollary}[theorem]{Corollary}

\theoremstyle{definition}
\newtheorem{definition}[theorem]{Definition}

\newtheorem{probl}[theorem]{Problem}

\theoremstyle{remark}
\newtheorem{rem}[theorem]{Remark}
\usepackage{hyperref}

\numberwithin{equation}{section}

\newcommand{\NN}{\mathbb{N}}
\newcommand{\RR}{\mathbb{R}}
\def\x{\mathbf{x}}
\def\P{\mathcal{P}}
\def\Q{\mathcal{Q}}

\def\S{\mathcal{S}}

\begin{document}

\title[Infinite dimensional moment problem: open questions and applications.]{Infinite dimensional moment problem: \\ open questions and applications.}

\author[M. Infusino]{Maria Infusino}
\address{Fachbereich Mathematik und Statistik,
Universit\"at Konstanz,\newline \indent
78457 Konstanz, Germany.}
\email{maria.infusino@uni-konstanz.de}
\thanks{The first author was partially supported by the Kongress--und--Vortragsreisenprogramm of DAAD (German Academic Exchange Service) and by the Equal Opportunity Council of the University of Konstanz within the project Konstanz Women in Mathematics}

\author[S. Kuhlmann]{Salma Kuhlmann}
\address{Fachbereich Mathematik und Statistik,
Universit\"at Konstanz,\newline \indent
78457 Konstanz, Germany.}
\email{salma.kuhlmann@uni-konstanz.de}
\thanks{The second author was partially supported by the AFF (Ausschuss f\"ur Forschungsfragen) research funding of University of Konstanz.}


\dedicatory{This paper is dedicated to Murray Marshall, who posed many of the questions here addressed and was still working on them in the very last days of his life.  We lost a wonderful collaborator and a dear friend. We sorely miss him.}


\begin{abstract} 
Infinite dimensional moment problems have a long history in diverse applied areas dealing with the analysis of complex systems but progress is hindered by the lack of a general understanding of the mathematical structure behind them. Therefore, such problems have recently got great attention in real algebraic geometry also because of their deep connection to the finite dimensional case. 
In particular, our most recent collaboration with Murray Marshall and Mehdi Ghasemi about the infinite dimensional moment problem on symmetric algebras of locally convex spaces revealed intriguing questions and relations between real algebraic geometry, functional and harmonic analysis. Motivated by this promising interaction, the principal goal of this paper is to identify the main current challenges in the theory of the infinite dimensional moment problem and to highlight their impact in applied areas. The last advances achieved in this emerging field and briefly reviewed throughout this paper led us to several open questions which we outline here.
\end{abstract}

\maketitle
 \section*{Introduction}\label{Intro}
The classical full $n-$dimensional $K-$moment problem asks whether a linear
functional $L : \mathbb{R}[x_1,\ldots, x_n] \rightarrow \mathbb{R}$ can be represented as integral w.r.t.\! some non-negative Radon measure supported on a fixed closed subsed $K$ of ${\mathbb{R}}^n$ (here $n\in\mathbb{N}$). When the starting functional is defined only for polynomials up to a certain degree, the moment problems is referred to as truncated. Already at an early stage the moment problem was generalized to the case of infinitely many variables, allowing the support of the measure to be possibly an infinite dimensional space (see e.g.~\cite{BK}, \cite{BS}, \cite{Bor-Yng75}, \cite{Heg75}, \cite{Lenard}, \cite{Pow71-74}, \cite{Schmu90}). Moment problems posed in infinite dimensional settings are often addressed as \emph{infinite dimensional moment problems}.

Infinite dimensional moment problems naturally arise in diverse applied areas dealing with the analysis of complex systems, i.e. many-body systems such as liquid composed of molecules, a molecule composed of atoms, a galaxy composed of stars. Since such a system consists of a huge number of identical components, the essence of its investigation is to evaluate selected characteristics (usually correlation functions), which encode the most relevant properties of the system. It is therefore fundamental to understand whether a finite number of given candidate correlation functions actually represent the correlation functions of some random distribution. This problem is well-known as \emph{realizability problem} and has a long history in statistical physics and in theoretical chemistry. The mathematical formulation of the realizability problem was given in \cite{Perc64} in order to better understand the closure problem in the theory of classical fluids and since then there has been an extensive production on realizability problems in statistical mechanics. Recent approaches are based on the interpretation of these problems as infinite dimensional truncated moment problems (see e.g.\!~\cite{IK}, \cite{IKR}, \cite{KLS07}, \cite{KLS09}). This new point of view was exploited not only in the analysis of interacting particle systems but also in the study of the realizability problem for random closed sets (e.g.~\cite{MolLach15}) and in material sciences in relation to random packing and heterogeneous materials modeling (see e.g.\! \cite{CraToSt03},  \cite{To02}, \cite{ToSt06}). The quantum mechanical variant of the realizability problem, known as  \emph{representability problem} for reduced density matrices, revealed to be a really useful approach to the computation of the ground state energies of molecules (e.g.~\cite{Co63}, \cite{GaPerc64}, \cite{K67}) actually yielding rigorous lower bounds. Recently, advances in computing power and in algorithms for positive semi-definite programming have led to an accuracy superior to that of the traditional electronic structure method, increasing the interest in the representability problem. Indeed, the quality of the approximation of the ground state energy depends on the availability of explicit representability criteria. 

The fundamental mathematical challenge which is behind each of these applications is therefore to derive treatable necessary and sufficient solvability conditions for the moment problem posed in a setting general enough to include the specific one of the considered application. This kind of structural investigation has been recently started and brought to new theoretical developments on the infinite dimensional full moment problem, which are essentially based on the interplay between the finite and the infinite dimensional case (see e.g. \cite{AJK}, \cite{GIKM}, \cite{GKM-IMP}, \cite{GKM}, \cite{IKR}, \cite{KLS07}, \cite{KLS09}). In particular, the results in \cite{GIKM} revealed intriguing relations between real algebraic geometry, functional and harmonic analysis, leading us to several open questions which we are going to present in this paper.

In Section 1 we introduce the general set up of the moment problem for linear functionals on any unital commutative $\RR-$algebra and briefly present some of the key-results obtained for this abstract formulation of the moment problem. In Section 2 we consider this problem for unital commutative $\RR-$algebras which are endowed with a locally multiplicatively convex topology. In particular, we review the results in \cite{GKM} for this class of moment problems and show some consequences of these results which brought us to interesting open questions. In Section 3 we focus on the special case of linear functionals on the symmetric algebra of a real locally convex space, since this setting is general enough to include the classical finite dimensional moment problem and several of the infinite dimensional moment problems appearing in the applications mentioned above. More precisely, we will review the results of \cite{GIKM} and compare them with some previous results \cite[Vol. II, Chapter~5, Section~2]{BK}, \cite{BS}, \cite[Section 4]{Bor-Yng75}, \cite[Section 3]{I}, \cite{IKR} about the moment problem on the symmetric algebra of a real locally convex (lc) space which is also nuclear. This comparison will actually be the source of several questions which to the best of our knowledge are still unsolved.

\section{Preliminaries}
Given a unital commutative $\mathbf{\RR-}$algebra $A$, we denote by $X(A)$ the \emph{character space} of $A$, i.e. the set $Hom(A; \RR)$ of all $\RR-$algebra homomorphisms $\alpha: A\to\RR$. For any $a \in A$, we define the \emph{Gelfand transform} $\hat{a} : X(A) \rightarrow \mathbb{R}$ as $\hat{a}(\alpha) := \alpha(a)$, $\forall\alpha\in X(A)$. We endow the character space $X(A)$ with the weakest topology $\tau_{X(A)}$ s.t. all Gelfand transforms are continuous, i.e. $\hat{a}$ is continuous for all $a \in A$. Note that $X(A)$ can be also seen as a subset of $\RR^A$ via the embedding
$\pi: X(A)\to \RR^A$ defined by $\pi(\alpha):=\left(\alpha(a)\right)_{a\in A}=\left(\hat{a}(\alpha)\right)_{a\in A}$,
and so it carries the natural topology induced by $\pi$ when $\RR^A$ is endowed with the product topology. It can be showed that this topology on $X(A)$ coincides with $\tau_{X(A)}$ (see \cite[Section 5.7]{M}).

\begin{probl}[The $KMP$ for unital commutative $\RR-$algebras]\label{GenKMP}\ \\
Given a closed subset $K\subseteq X(A)$ and a linear functional $L: A\to\RR$, does there exist a non-negative Radon measure $\mu$ on $X(A)$ such that 
$$L(a)=\int_{X(A)}\hat{a}(\alpha) \mu(d\alpha), \forall a\in A\quad\text{and}\quad supp(\mu)\subseteq K?$$
\end{probl}
If the answer is positive then we say that $\mu$ is a \emph{$K-$representing measure} for $L$ and that $L$ is represented by $\mu$ on $K$.
Recall that a Radon measure on a Hausdorff topological space $X$ is a non-negative measure on the $\sigma-$algebra of Borel sets of $X$ that is locally finite and inner regular.

Note that for $A=\RR[\x]=\RR[x_1,\ldots, x_d]$ Problem \ref{GenKMP} reduces to the classical $d-$dimensional $K-$moment problem. Indeed, using that the identity is the unique $\RR-$algebra homomorphism from $\RR$ to $\RR$, it can be proved that any $\RR$-algebra homomorphism from $\RR[\x]$ to $\RR$ corresponds to a point evaluation $p\mapsto p(\alpha)$ with $\alpha\in\RR^d$ and so $X(\RR[\x])$ is topologically isomorphic to $\RR^d$ (see e.g. \cite[Proposition 5.4.5]{M}).

As for the finite dimensional KMP, a necessary and sufficient condition for the existence of a solution is provided by a generalized version of the so-called Riesz-Haviland theorem. 
\begin{theorem}[Generalized Riesz-Haviland Theorem]\label{GRHThm}
Let $K\subseteq X(A)$ be closed and $L: A\to\RR$ linear. Suppose that there exists $p\in A$ s.t. $\hat{p}\geq 0$ on $K$ and for each $i\in\NN$ the set $\{\alpha\in K: \hat{p}(\alpha)\leq i\}$ is compact. Then: \\ $L$ has a $K-$representing measure if and only if $L(\operatorname{Pos}(K))\subseteq[0, +\infty)$, where
\begin{equation}\label{PosK}
\operatorname{Pos}(K):= \{ a \in A : \hat{a} \ge 0 \text{ on } K\}.
\end{equation}
\end{theorem}
The characterization given in Theorem \ref{GRHThm} actually holds under slightly more general assumptions (c.f.\!\! \cite[Section~3.2]{M}), but the ones stated here are general enough for our purposes in this article. 

Theorem \ref{GRHThm} reduces Problem \ref{GenKMP} to the problem of characterizing $\operatorname{Pos}(K)$. However, the latter is seldom finitely generated (see \cite[Proposition 6.1]{Scheid99}) and in general there is no practical decision procedure to check whether an element of the algebra belongs to $\operatorname{Pos}(K)$. A common approach in finite and infinite dimensions to attack this problem is to try to approximate elements in $\operatorname{Pos}(K)$ with elements of $A$ whose Gelfand transform is ``more evidently'' non-negative, e.g. sum of even powers of elements of $A$ (see discussion about Question \ref{pos-certificate}). In this spirit it is somehow natural to consider $2d-$power modules of the algebra $A$ for $d\in\NN$. 

\begin{definition}
Let $d\in\NN$. A \emph{$2d$-power module} of $A$ is a subset $M$ of $A$ satisfying $1 \in M, \ M+M \subseteq M \text{ and } a^{2d}M \subseteq M \text{ for each } a \in A.$ 
\end{definition}
In the case $d=1$, $2d$-power modules are referred to as \emph{quadratic modules}. We denote by $\sum A^{2d}$ the set of all finite sums $\sum a_i^{2d}$, $a_i \in A$. $\sum A^{2d}$ is the smallest $2d$-power module of $A$. 
\begin{definition}
Let $\{p_j\}_{j\in J}$ be an arbitrary subset of elements in $A$. The $2d$-power module of $A$ generated by $\{p_j\}_{j\in J}$ is defined as $M = \sum A^{2d}+\sum_{j\in J}\left(p_j\sum A^{2d}\right).$ 
\end{definition}
Note that the index set $J$ can have also infinite cardinality. In the classical finite dimensional moment problem only finitely generated quadratic modules have been considered. The results presented in this paper also hold for infinitely generated $2d-$power modules both in the classical and in the general setting.

A linear functional $L : A \rightarrow \mathbb{R}$ is said to be
{\it positive} if $L(\sum A^{2d}) \subseteq [0,\infty)$ and {\it
$M-$positive} for some $2d-$power module $M$ of $A$, if $L(M)\subseteq
[0,\infty)$. For any subset $M$ of $A$, we set 
$$X_M:= \{ \alpha \in X(A) :
\hat{a}(\alpha)\ge 0 ,\ \forall a\in M\},$$  
which is closed in $\left(X(A), \tau_{X(A)}\right)$.
If $M= \sum A^{2d}$ then
$X_M = X(A)$. If $M$ is the $2d$-power module of $A$ generated by
$\{p_j\}_{j\in J}$ then $X_M\!=\!\{\alpha\!\in\! X(A):\hat{p_j}(\alpha)\ge 0,\ \forall j\!\in\! J\}.$

Given a $2d-$power module $M$ of $A$ and a linear functional $L:A\to\RR$, it is clear that if there exists a $X_M-$representing measure $\mu$ for $L$, then $L$ is $M-$positive since $M\subseteq \operatorname{Pos}(X_M)$. This condition was proved to be also sufficient when the $2d-$power module is Archimedean, i.e.\! 
for each $a \in A$ there exists an integer $N$ such that $N \pm a
\in M$, by mean of the Jacobi Positivstellensatz (see \cite{Jac}). Note that this result was already known for Archimedean quadratic modules in $\RR[\x]$ thanks to the Putinar Positivstellensatz \cite{Put93}.
\begin{theorem}\label{Archm-thm}
Let $M$ be an Archimedean $2d$-power module of $A$ and $L : A \rightarrow \mathbb{R}$ a linear functional. $L$ has a $X_M-$representing measure iff $L(M)\subseteq[0,+\infty)$.
\end{theorem}
\begin{proof} See  \cite[Corollary 2.6]{GMW}. The conclusion can be also obtained as a consequence of \cite[Theorem 5.5]{GK}.
\end{proof}

In the case of non-Archimedean $2d-$power module $M$, it revealed to be extremely powerful to consider topologies on $A$ such that
\begin{equation}\label{clos-cond}
\operatorname{Pos}(X_M)\subseteq \overline{M}^{\tau}.
\end{equation}
Indeed, if this condition holds then an application of the
Hahn-Banach theorem together with Theorem \ref{GRHThm} shows that:
$L$ has a $X_M-$representing measure if and only if $L$ is $\tau-$continuous and $M-$positive.
For the case $A=\RR[\x]$, there are several papers dealing with the
moment problem for linear functionals continuous w.r.t. certain weighted norm topologies (see e.g. \cite{BCR}, \cite{GKS}, \cite{GMW}, \cite{L}, \cite{Schmu78}). This approach has been extended to more general settings (see e.g. \cite{BCR-book}, \cite{BM}, \cite{Schmu78}) and recently also to the one of Problem \ref{GenKMP} in \cite{GK}, \cite{GKM}, \cite{GMW}, where the authors analyze integral representations of linear functionals continuous w.r.t. locally multiplicatively convex (lmc) topologies.

\section{The moment problem on lmc topological $\RR$-algebras}
In this section we shortly present some closure results of the type \eqref{clos-cond} and their corresponding solutions for Problem \ref{GenKMP} which have been recently developed for locally multiplicatively convex $\RR-$algebras (see \cite{GK} and \cite{GKM}). In particular we are going to highlight some consequences of these results which led us to interesting open questions.

Let $(A, \cdot)$ be a unital commutative $\RR-$algebra and $\sigma$ be a submultiplicative seminorm on $A$, i.e. $\sigma (a \cdot b)\leq \sigma(a)\sigma(b)$ for all $a,b\in A$. The algebra $A$ together with such a $\sigma$ is called a \emph{submultiplicative seminormed $\RR-$algebra} and is denoted by $(A, \sigma)$. We denote the set of all $\sigma-$continuous $\RR-$algebra
homomorphisms from $A$ to $\RR$ by $sp(\sigma)$, which we refer to as the \emph{Gelfand spectrum} of $(A, \sigma)$, i.e.
$$\mathfrak{sp}(\sigma) := \{ \alpha \in X(A): \alpha \text{ is } \sigma-\text{continuous}\}.$$
We endow $\mathfrak{sp}(\sigma)$ with the subspace topology induced by $\left(X(A), \tau_{X(A)}\right)$.
\begin{lemma}\label{char-GelfSp} \cite[Lemma 3.2, Corollary 3.3]{GKM} 
For any submultiplicative seminormed $\RR-$algebra $(A, \sigma)$ we have that
$$
\mathfrak{sp}(\sigma)=\{ \alpha \in X(A) : |\alpha(a)|\le \sigma(a) \text{ for all } a\in A\}.
$$
and $\mathfrak{sp}(\sigma)$ is compact in $(X(A), \tau_{X(A)})$.
\end{lemma}

We present here a closure result of the type \eqref{clos-cond} for submultiplicative seminormed $\RR-$algebras which was proved in \cite[Theorem 3.7]{GKM} and which allows to get new necessary and sufficient conditions to solve Problem \ref{GenKMP}.
\begin{theorem}\label{ClosureRes}
Let $(A, \sigma)$ be a submultiplicative seminormed $\RR-$algebra, $d\in\NN$ and $M$ a $2d$-power module of $A$ (not necessarily Archimedean). Then
$$\overline{M}^{\sigma} = \operatorname{Pos}(X_M \cap  \mathfrak{sp}(\sigma)).$$ 
\end{theorem}
Using Theorem \ref{ClosureRes} and Theorem \ref{GRHThm}
\footnote{
Note that we can apply Theorem  \ref{GRHThm} for $K=X_M \cap  \mathfrak{sp}(\sigma)$. Indeed, since $X_M$ is a closed subset of $\left(X(A), \tau_{X(A)}\right)$, we have by Lemma \ref{char-GelfSp} that $X_M \cap  \mathfrak{sp}(\sigma)$ is compact and so the assumption of Theorem  \ref{GRHThm} is fulfilled for $p=1$.}, one can derive the following result \cite[Corollary 3.8]{GKM} for the moment problem in this setting. 
\begin{prop}\label{coroll-SubAlg}
Let  $(A, \sigma)$ be a submultiplicative seminormed $\RR-$algebra, $d\in\NN$, $M$ a $2d$-power module of $A$ and $L : A \rightarrow \mathbb{R}$ a linear functional. $L$ has a representing measure supported on $X_M \cap  \mathfrak{sp}(\sigma)$ if and only if $L$ is $\sigma-$continuous and $L(M)\subseteq[0,+\infty)$.
\end{prop}

Theorem \ref{ClosureRes} also extends to the class of locally multiplicatively convex (lmc) topologies, i.e. topologies induced by families of submultiplicative seminorms. We will call \emph{lmc $\RR-$algebra} a unital commutative $\RR-$algebra $A$ endowed with a lmc topology. 
\begin{theorem}\label{ClosureRes-lmc}
Let  $(A, \tau)$ be a lcm $\RR-$algebra, $d\in\NN$ and $M$ a $2d$-power module of $A$. Then
$$\overline{M}^{\tau} = \operatorname{Pos}(X_M \cap  \mathfrak{sp}(\tau)).$$ 
\end{theorem}
Note that $\mathfrak{sp}(\tau)=\bigcup_{\rho\in \mathcal{S}} \mathfrak{sp}(\rho)$ where $\S$ is a directed family of seminorms generating $\tau$ (such a family always exists since $\tau$ is a lc topology). Recall that a family $\mathcal{S}$ of seminorms on $V$ is said to be \textit{directed} if $$\forall \ \rho_1,\rho_2 \in \mathcal{S}, \ \exists \ \rho \in \mathcal{S} \text{ and } C>0 \  \text{ s.t.} \ C\rho(v) \ge \max\{ \rho_1(v), \rho_2(v)\} \ \forall \ v\in V.$$

Using Theorem \ref{ClosureRes-lmc} and Theorem \ref{GRHThm}, we get
\begin{prop}\label{coroll-lmcAlg}
Let  $(A, \tau)$ be a lcm $\RR-$algebra, $d\in\NN$, $M$ a $2d$-power module of $A$ and $L : A \rightarrow \mathbb{R}$ a linear functional. $L$ has a representing measure supported on $X_M \cap  \mathfrak{sp}(\tau)$ if and only if $L$ is $\tau-$continuous and $L(M)\subseteq[0,+\infty)$.
\end{prop}
For certain topological unital commutative $\RR-$algebras and in particular for the algebra of polynomials, weaker assumptions than the continuity of $L$ are known to guarantee the equivalence between the non-negativity of $L$ on a quadratic module $M_2$ and the existence of a $X_{M_2}-$representing measure for $L$ (see Section \ref{sec:symmAlg} for more details). 
This motivates the following more general question.
\begin{q}\label{weaken-con}
Given a topological unital commutative $\RR$-algebra $(A, \tau)$ and a $2d$-power module $M$ of $A$, what are the weakest assumptions on an $M-$positive linear functional such that it admits a $X_M$-representing measure?\end{q}

Proposition \ref{coroll-lmcAlg} has also a direct application related to the problem of checking the $M_{2d}-$positivity of a linear functional for a given $2d-$power module $M_{2d}$. For the case $d=1$, it is well-known that given a quadratic module $M_2$ of $A$ the $M_2-$positivity of $L$ is equivalent to the positive semidefiniteness of a certain sequence of infinite Hankel matrices associated to $L$ and $M_2$. In particular, for the truncated moment problem the $M_2-$positivity of $L$ can be effectively checked through semidefinite programming techniques. 
\begin{q}\label{pos-certificate} Is there a similar criterion for the non-negativity of a linear functional on an arbitrary $2d-$power module with $d\geq 2$? 
\end{q}

Clearly, if $L$ is non-negative on a quadratic module $M_2$ then it is non-negative on any $2d-$power module $M_{2d}$ having the same set of generators as $M_2$ for any $d\geq 2$ (as $M_{2d}\subset M_2$). Thus the positive semidefiniteness of the sequence of Hankel matrices associated to $L$ and $M_2$ is a sufficient condition for the $M_{2d}-$positivity of $L$ for any $d\geq 2$. However, we do not know if/when this is also necessary.

For the case $M_{2d}=\sum A^{2d}$ we have the following criterion under the additional assumption of continuity of the considered functional.
\begin{corollary}\label{cor-powers}
Let  $(A, \tau)$ be a lmc $\RR-$algebra and $L : A \rightarrow \mathbb{R}$ a $\tau-$continuous linear functional. The following are equivalent:
\begin{enumerate}
\item $L$ is positive, i.e. $L(\sum A^{2})\subseteq[0,+\infty)$.
\item $\forall\, d\in\NN$, $L(\sum A^{2d})\subseteq[0,+\infty)$.
\item $\exists\, d\in\NN$, $L(\sum A^{2d})\subseteq[0,+\infty)$.
\end{enumerate}
\end{corollary}
\proof
From the discussion immediately after Question \ref{pos-certificate} we know that it is always true that: (1) $\Rightarrow$ (2) $\Rightarrow$ (3). Let us show that also (3) $\Rightarrow$ (1). Suppose that $\exists\, d\in\NN$, $L(\sum A^{2d})\subseteq[0,+\infty)$. Then, by Proposition \ref{coroll-lmcAlg} applied to $M= \sum A^{2d}$,  $L$ has a representing measure $\mu$ supported on $X_M \cap  \mathfrak{sp}(\tau)$. (In this case $X_M=X(A)$ and so $supp(\mu)\subseteq \mathfrak{sp}(\tau)$.) In particular, this implies that for all $a\in \sum A^{2}$ we have
$L(a)=\int_{\mathfrak{sp}(\tau)} \hat{a}(\alpha)\mu(d\alpha)$
which is non-negative since $\mu$ is non-negative and $\hat{a}(\alpha)\geq 0$ for all $a\in \sum A^{2}$. Hence, (1) holds.\endproof
Note that while (1) $\Rightarrow$ (2) $\Rightarrow$ (3) holds in general, to show that (3) $\Rightarrow$ (1) we have used the continuity of $L$. This motivates the following question:
\begin{q}\label{pos-certificate2}
Does Corollary~\ref{cor-powers} still hold if we remove or weaken the continuity assumption?
\end{q}
Let us observe that once we have a positive answer to Question \ref{weaken-con} (i.e.\! once we have identified some conditions weaker than continuity under which the non-negativity of $L$ on some $2d-$power module implies the existence of a representing measure for $L$), then the equivalence of (1), (2) and (3) can be established along the same lines of the proof of Corollary \ref{cor-powers}. However, these same methods do not provide an answer to the following question.
\begin{q}\label{pos-certificate-modules}
Can Corollary~\ref{cor-powers} be extended to $2d-$power modules of $A$ having the same non-empty set of generators?
\end{q}

Theorem \ref{ClosureRes-lmc} can be viewed as a strengthening in the
commutative case of the result in \cite[Lemma 6.1 and Proposition 6.2]{Schmu78} for enveloping algebras. 

Having an eye on applications of the moment problem in other fields, we were naturally led to consider special kind of topological $\RR-$algebras and to investigate to which extent is possible to improve the results presented in this section for these particular cases. 

\section{The moment problem on symmetric algebras of lc real spaces}\label{sec:symmAlg}
In this section we focus on Problem \ref{GenKMP} when $A$ is the symmetric (tensor) algebra $S(V)$ of a vector space $V$ over $\RR$, i.e., the tensor algebra $T(V)$ factored by the ideal generated by the elements $v\otimes w -w\otimes v$, $v,w\in V$. If we fix a basis $x_i$, $i\in \Omega$ of $V$, then  $S(V)$ is identified with the polynomial ring $\mathbb{R}[x_i : i\in \Omega]$, i.e., the free $\mathbb{R}$-algebra in commuting variables $x_i$, $i\in \Omega$. For any integer $k\ge 0$, denote by $S(V)_k$ the $k$-th homogeneous part of $S(V)$, i.e.\! the image of $k$-th homogeneous part $V^{\otimes k}$ of $T(V)$ under the canonical map $\sum_{i=1}^n f_{i1}\otimes \cdots \otimes f_{ik} \mapsto \sum_{i=1}^n  f_{i1}\cdots f_{ik}$ with $f_{ij} \in V$ for  $i=1,\dots, n$, $j=1,\dots,k$, $n\ge 1$. Note that $S(V)_0 = \mathbb{R}$ and $S(V)_1 = V$.

Since $S(V)$ is in effect isomorphic to the polynomial ring having as variables the vectors of a basis for $V$, the symmetric algebra has always been a natural choice for the algebra on which posing the moment problem especially in infinite dimensional contexts. Indeed, since the early days of the moment theory,  several problems appearing in applied fields (e.g. statistical mechanics, quantum field theory, etc.) have been modeled as moment problems for linear functionals defined on the algebra polynomials in infinitely many variables whose representing measures are usually required to have support on an infinite-dimensional space (see the introduction for more references). Hence, the choice of considering Problem \ref{GenKMP} for the symmetric algebra of a general vector space $V$ (possibly infinite dimensional) is general enough to encompass a multitude of practical applications and clearly includes the classical case (when $V$ is finite dimensional). An early systematic study of the moment problem for linear functionals on the symmetric algebra of a locally convex nuclear space can be found e.g. in \cite[Chapter 5, Section 2]{BK}, \cite{BS}, \cite{Bor-Yng75}, \cite{Heg75}, \cite{Pow71-74}, \cite[Section~12.5]{Schmu90}. More recently, new advances were obtained both for specific choices of the locally convex nuclear space $V$ (see \cite{IK}, \cite{IKR}) and for the general case of the algebra of polynomials in an arbitrary set of variables $\{x_i;\> i \in \Omega\}$ (see \cite{AJK}, \cite{GKM-IMP}).

Considering the general results presented in the previous section, it seemed natural to look for lmc topologies on the symmetric algebra in order to be able to apply Proposition \ref{coroll-lmcAlg} to this kind of algebras.  In \cite{GIKM} it is indeed explained how a lc topology $\tau$ on a real vector space $V$ can be extended to a lmc topology $\overline{\tau}$ on the symmetric algebra $S(V)$ and a complete criterion for the existence of a solution to the $K$-moment problem for $\bar{\tau}$-continuous linear functionals on $S(V)$ is given. Furthermore, \cite{GIKM} contains a detailed comparison between these results and the ones in \cite[Theorem 2.1]{BK}, \cite{BS}, \cite[Theorem 2.3]{IKR} for the $K$-moment problem on locally convex nuclear spaces. Starting from this comparison, we outline some connections to further previous works on the moment problem on symmetric algebras of lc nuclear spaces and pose some questions which naturally emerge from this analysis.

We start by briefly reporting the main results of \cite{GIKM} in the case $V$ is a $\RR-$vector space with the topology given by a single seminorm $\rho$. The procedure to extend $\rho$ on $V$ to a submultiplicative seminorm $\overline{\rho}$ on $S(V)$ can be summarized as follows.
\begin{enumerate} 
\item For $k\in\NN$, let us consider the projective tensor seminorm on $V^{\otimes k}$, i.e.
$$
 \rho^{\otimes k}(g):= \inf\left\{ \sum_{i=1}^N \rho(g_{i1}) \cdots \rho_k(g_{ik}) : g= \sum_{i=1}^N g_{i1}\otimes \cdots \otimes g_{ik}, \ g_{ij} \in V, \ N\in\NN\right\}.
$$
\item Denote by $\pi_k : V^{\otimes k} \rightarrow S(V)_k$ the quotient map $\pi$ restricted to $V^{\otimes k}$ and
define $\overline{\rho}_k$ to be the quotient seminorm on $S(V)_k$ induced by $\rho^{\otimes k}$, i.e.  
\begin{align*}\overline{\rho}_k(f) :=& \inf\{ \rho^{\otimes k}(g) : g \in V^{\otimes k}, \ \pi_k(g)=f\} \\ 
=&  \inf\left\{ \sum_{i=1}^N \rho(f_{i1})\cdots \rho(f_{ik}) : f = \sum_{i=1}^N f_{i1}\cdots f_{ik}, f_{ij} \in V, N\in\NN \right\}.\end{align*}
Define $\overline{\rho}_0$ to be the usual absolute value on $\RR$.
\item For any $h\in S(V)$, say $h= h_0+\dots + h_{\ell}$, $h_k \in S(V)_k$, $k=0,\dots,\ell$, define $$\overline{\rho}(h) := \sum_{k=0}^{\ell} \overline{\rho}_k(h_k).$$
\end{enumerate}
The seminorm $\overline{\rho}$ is called the \emph{projective extension of $\rho$ to $S(V)$} and in \cite[Proposition~3.2]{GIKM} it is proved to be submultiplicative, i.e. $\overline{\rho}(f\cdot g)\leq\overline{\rho}(f)\overline{\rho}(g),\,\,\forall f,g\in S(V)$. Proposition \ref{coroll-SubAlg} can be therefore applied and, exploiting the universal property of the symmetric algebra, it is possible to show the following result.

\begin{prop}{\cite[Corollary 3.7]{GIKM}} \label{PropB} 
Let $(V,\rho)$ be a seminormed $\mathbb{R}$-vector space, $M$ a $2d$-power module of $S(V)$ and 
$L : S(V)\rightarrow \mathbb{R}$ a linear functional. $L$ is $\overline{\rho}$-continuous and 
$L(M)\subseteq [0,+\infty)$ if and only if $\exists\, !\, \mu$ on $V^*$ such that $$L(f)=\int\limits_{V^*}\hat{f}(\alpha) \mu(d\alpha)\quad\text{ and}\quad \operatorname{supp}\mu\subseteq X_M\cap \overline{B}_1^{\|\cdot\|_{\rho}},$$
where $V^*$ denotes the algebraic dual of $V$, $\|\cdot\|_{\rho}$ is the operator norm on $V^*$, i.e. $\|\beta\|_{\rho}:=\sup\limits_{\stackrel{v\in V}{\rho(v)\leq 1}}|\beta(v)|$
and $\overline{B}_1^{\|\cdot\|_{\rho}}:= \{ \beta\in V^*: \|\beta\|_{\rho}\leq 1\}.$ 
\end{prop}

\begin{rem}\label{rem-ineq-ext}
For any two seminorms $\rho_1$ and $\rho_2$ on $V$, we write $\rho_1 \succeq \rho_2$ to indicate that there exists $c>0$ such that $c\rho_1(v) \ge \rho_2(v), \ \forall \ v\in V.$ It is important to note that $\rho_1 \succeq \rho_2$ does not imply in general that $\overline{\rho_1} \succeq \overline{\rho_2}$. However, \cite[Proposition~3.4]{GIKM} guarantees that there exists $C>0$ such that $\overline{C\rho_1} \succeq \overline{\rho_2}$. \end{rem}

Let us consider now the case when $\tau$ is any locally convex topology on a $\RR-$vector space $V$ and let $\P$ be a directed family of seminorms generating $\tau$. In \cite{GIKM} the authors consider the topology $\overline{\tau}$ on $S(V)$ generated by the family of seminorms $\Q:=\{\overline{n\rho}:  \rho\in \P,\,\,n\in\NN\}$, which is directed in virtue of Remark \ref{rem-ineq-ext}. Hence, by using Proposition \ref{coroll-lmcAlg} and the universal property of the symmetric algebra, it is possible show the following result.
\begin{prop}\label{PropC}
Let $(V,\tau)$ be a lc topological vector space over $\RR$ whose topology is generated by a directed family of seminorms $\P$. Let $M$ be a $2d$-power module of $S(V)$ and 
$L : S(V)\rightarrow \mathbb{R}$ a linear functional. $L$ is $\overline{\tau}$-continuous and 
$L(M)\subseteq [0,+\infty)$ iff $\exists\, !\, \mu$ on $V^*$: $L(f)=\int\limits_{V^*}\hat{f}(\alpha) \mu(d\alpha)$ and $\operatorname{supp}\mu\subseteq X_M\cap \overline{B}_n^{\|\cdot\|_{\rho}}$
for some $n\in\NN$ and $\rho\in \P$.
\end{prop}

Note that $\overline{\tau}$ is the finest lmc topology $\overline{\tau}$ on $S(V)$ extending $\tau$ (see \cite[Proposition~5.1]{GIKM}). Therefore, Proposition \ref{PropC} holds for all linear functionals on $S(V)$ which are continuous w.r.t.\! any other lmc topology on $S(V)$ extending $\tau$. However, it is not clear to us how to construct further topologies on $S(V)$ fulfilling these properties. For instance, we already do not know if it is possible to extend $\tau$ to a lmc topology on $S(V)$ using the injective topology on the tensor power $V^{\otimes k}$ instead of the projective one as in \cite{GIKM}.
\begin{q}
Is it possible to give a characterization of all the lmc topologies on $S(V)$ extending $\tau$?
\end{q}
More in general, one could ask if there exists any other topology out of this class for which a result of the kind in Proposition \ref{PropC} would be still true, i.e.
\begin{q}\label{all-top}
Can we characterize all the topologies on $S(V)$ for which an analogous result to Proposition \ref{PropC} holds?
\end{q}

The assumption of $\bar{\tau}$-continuous is actually pretty strong. Already for the finite dimensional case $V=\RR^d$ analogues of Proposition \ref{PropC} hold under weaker assumptions than the continuity of $L$. Indeed, in \cite[Theorem 10]{Nuss65} Nussbaum proved that whenever $L: \RR[x_1, \ldots, x_d]\to\RR$ fulfills the so-called Carleman's condition i.e.
\begin{equation}\label{mult-Carleman}
  \sum\limits_{n=1}^{\infty}{L(x_j^{2n})}^{-\frac{1}{2n}}=\infty,\quad \forall\, j=1,\dots,d,
\end{equation}
then $L$ positive is equivalent to the existence of a representing measure for $L$ supported on $\RR^d$. Note that if $L$ is continuous w.r.t.\! any lmc topology $\omega$ on $\RR[x_1, \ldots, x_d]$ then \eqref{mult-Carleman} holds. In fact, if $\mathcal{S}$ is a directed family of submultiplicative seminorms generating $\omega$ then the $\omega-$continuity of $L$ implies there exist $\rho \in \mathcal{S}$ and $C>0$ such that
 $|L(f)|\le C \rho(f)$ for all $f\in \RR[x_1, \ldots, x_d]$.
In particular for any $n\ge 1$ and any $j=1,\ldots, d$, we get $|L(x_j^{2n})| \le C \rho(x_j)^{2n}$ which implies \eqref{mult-Carleman}.

Nussbaum's result still holds under weaker conditions than \eqref{mult-Carleman} as showed in \cite[Theorem 4.9]{M-I}, \cite[Theorem 0.1]{M-II} and \cite[Proposition 1]{Schmu91} (see also \cite[Lemma 0.2 and Theorem 0.3]{M-II} for more details about the relation between these conditions). Using Nussbaum's result (resp.\! its stronger version \cite[Theorem 0.1]{M-II}), in \cite[Theorem 4.1]{IKR} (resp. \cite[Corollary 0.6]{M-II}) it is showed that for any linear functional $L$ on $\RR[x_1, \ldots, x_d]$ satisfying \eqref{mult-Carleman} (resp. equation (0.1) in \cite[Theorem 0.1]{M-II}) the non-negativity of $L$ on a quadratic module $M_2$ (even uncountably generated) of $\RR[x_1, \ldots, x_d]$ is equivalent to the existence of a $X_{M_2}-$representing measure for $L$. 
The results mentioned above have been generalized to the case of infinitely many variables in \cite[Section 4 and 5]{GKM-IMP}. In particular, the corresponding results for linear functionals non-negative on quadratic modules of $\RR[x_i, i\in \Omega]$ with $\Omega$ possibly uncountable are \cite[Theorem 5.1 and~5.4]{GKM-IMP}, which hold under further assumptions on the quadratic modules. Note that in the uncountable case the representing measures provided by these results in \cite{GKM-IMP} are not necessarily Radon. 

Similar results were obtained in \cite[Vol. II, Chapter 5, Section 2]{BK}, \cite{BS}, and \cite{IKR}, \cite[Section 3]{I}  when $V$ belongs to a certain class of nuclear spaces. Namely, $(V, \tau)$ is assumed to be: separable, the projective limit of a family $(H_j)_{j\in J}$ of Hilbert spaces ($J$ is an index set containing~$0$) which is directed by topological embedding such that each $H_j$ is embedded topologically into $H_0$, and nuclear, i.e.\! for each $j_1\in J$ there exists $j_2\in J$ such that the embedding $H_{j_2}\hookrightarrow H_{j_1}$ is quasi-nuclear. Thus $\tau$ is the locally convex topology on $V$ induced by the directed family $\P$ of norms on $V$, where $\P$ consists of the norms on $V$ which are induced by the embeddings $V \hookrightarrow H_j$, $j \in J$. The topology $\tau$ is usually referred to as the \it projective topology \rm on $V$ and it is clearly a Hausdorff topology. In this setting Berezansky, Kondratiev, and \v Sifrin proved in \cite[Vol. II, Theorem~2.1]{BK} and \cite{BS} a Nussbaum's type result, which we restate here using the formulation given in \cite[Theorem 6.1]{GIKM}.
\begin{theorem} \label{nuclear} Let $(V,\tau)$ be a nuclear space of the special sort described above and $L : S(V) \rightarrow \mathbb{R}$ be a linear functional. Assume:

(1) for each $k\ge 0$ the restriction map $L : S(V)_k \rightarrow \mathbb{R}$ is continuous with respect to the locally convex
topology $\overline{\tau}_k$
on $S(V)_k$ induced by the norms $\{\overline{\rho}_k : \rho \in \mathcal{P}\}$; and

(2) there exists a countable subset $E$ of $V$ whose linear span is dense in $(V,\tau)$ such that
$$ m_0 := \sqrt{L(1)}, \text{ and } m_k := \sqrt{\sup_{f_1, \dots ,f_{2k} \in E} |L(f_1\dots f_{2k})|}, \text{ for } k \ge 1$$ are finite and the class $C\{ m_k\}$
is quasi-analytic.

Then  $L$ is positive, i.e.\! $L(\sum S(V)^2) \subseteq [0,\infty)$, if and only if there exists a non-negative Radon measure $\mu$ on the dual space $V^*$ supported by the topological dual $V'$ of $(V,\tau)$ such that $L(f)= \int \hat{f} d\mu$ $\forall$ $f\in S(V)$.
\end{theorem}

We will refer to Conditions (1) and (2) combined as \it determining \rm condition following the notation of  \cite{I}, \cite{IKM} and \cite{IKR}. For a comparison between the determining condition, the infinite dimensional version of Carleman's condition and the other conditions appearing in \cite{GKM-IMP} see \cite[Remark 3.2]{IKM}. In particular, if $L$ is $\overline{\tau}$-continuous then $L$ is determining (see \cite[Remark 6.2, (9)-(10)]{GIKM} for more details). Therefore, when we restrict our attention to this special class of locally convex nuclear spaces and to positive functionals, Proposition~\ref{PropC} is less general than Theorem~\ref{nuclear}. However, as pointed out in \cite[Remark 6.2 (8), (11)]{GIKM}, Proposition~\ref{PropC} is more general than Theorem~\ref{nuclear} because it holds for any locally convex topological space (not just separable and nuclear) and for arbitrary $2d$-power modules (not just for $\sum S(V)^2$) providing in this way better information about the support of the representing measure. This comparison leads to the following question.
\begin{q}\label{weak-cont-nuclear1}
Is it possible to generalize Proposition \ref{PropC} by weakening the continuity assumption on $L: S(V)\to \RR$ without any further assumption on~$V$? How would this affect the support of the representing measure?
\end{q} 

An encouraging result in this direction is given by \cite[Theorem 2.3]{IKR} for the special case when $V = \mathcal{C}_c^{\infty}(\mathbb{R}^n)$, the set of all infinitely differentiable functions with compact support contained in $\mathbb{R}^n$. Note that $\mathcal{C}_c^{\infty}(\mathbb{R}^n)$ can be constructed as the projective limit of a family of weighted Sobolev spaces and so endowed with the corresponding projective topology $\tau_{proj}$, which actually makes it a nuclear space belonging to the class for which Theorem~\ref{nuclear} holds (see \cite{IKR} for more details). The choice of this particular space seems to be not so restrictive from the point of view of moment problems appearing in applied fields, which very often deal with representing measures supported on non-compact subsets of the space of generalized functions as the topological dual of $(\mathcal{C}_c^{\infty}(\mathbb{R}^n), \tau_{proj})$. For convenience, let us restate here \cite[Theorem 2.3]{IKR} in the notation used so far.
\begin{theorem}\label{IKR-Thm}
Let $V=\mathcal{C}_c^{\infty}(\mathbb{R}^n)$ be endowed with the topology $\tau_{proj}$ described above, $M_2$ be a quadratic module of $S(V)$ and 
$L : S(V)\rightarrow \mathbb{R}$ a linear functional. Assume that
$L$ is determining. Then 
$L(M_2)\subseteq [0,+\infty)$ iff $\exists\, !\, \mu$ on $V^*$: $L(f)=\int_{V^*}\hat{f}(\alpha) \mu(d\alpha)$ and $\operatorname{supp}\mu\subseteq X_{M_2}\cap V'$.
\end{theorem}
It is not clear to us if such a result would still hold for any $2d-$power module providing, at least for this particular choice of $V$, a generalization of Proposition \ref{PropC} and a partial positive answer to Question \ref{weak-cont-nuclear1}. Moreover, Theorem \ref{IKR-Thm} also covers the case of non-compactly supported representing measures while the continuity of $L$ in Proposition \ref{PropC} enforces the support of the representing measure to be compact. Therefore, it would be extremely interesting to get such a result for a larger class of spaces including $(\mathcal{C}_c^{\infty}(\mathbb{R}^n), \tau_{proj})$ and for any $2d-$power module rather than  just for $d=1$. Hence:
\begin{q}\label{weak-cont-nuclear2}
What are the weakest possible assumptions on $V$ s.t. Theorem~\ref{IKR-Thm} holds? Under these assumptions, would the result still hold for any $2d-$power module (not just for $d=1$)?
\end{q} 
The comparison of Proposition \ref{PropC} with Theorem~\ref{nuclear} and Theorem \ref{IKR-Thm} leads to the following general question (which includes Questions \ref{weak-cont-nuclear1} and \ref{weak-cont-nuclear2}).
\begin{q}\label{weak-cont-nuclear}
What are the weakest possible assumptions on the lc space $V$ and on the linear functional $L:S(V)\to\RR$ s.t. the nonnegativity of $L$ on a $2d-$power module $M$ is equivalent to the existence of a $X_M$-representing measure for $L$?
\end{q} 

As mentioned above there are several other papers in literature dealing with the moment problem on the symmetric algebra of a nuclear space. In particular, in \cite[Section 4]{Bor-Yng75} the authors consider a different class of nuclear spaces than the one described above; namely, $V$ is taken to be a real nuclear space which is a strict inductive limit of a countable family $(V_s)_{s\in \Omega}$ of its subspaces. For each $s\in \Omega$, they consider the tensor power $V_s^{\otimes k}$ is endowed with the $\pi-$topology and define $V^{\otimes k}$ as the inductive limit of $(V_s^{\otimes k})_{s\in \Omega}$ endowed with the corresponding inductive limit topology. This gives a natural quotient topology on $S(V)_k$ and $S(V)$ is endowed with the lc direct sum topology which we denote in the following by $\tau_{BY}$. For $\tau_{BY}-$continuous linear functionals on $S(V)$, Borchers and Yngvason proved a generalization of the classical Riesz-Haviland theorem, which we restated here in our notation.
\begin{theorem}
Let $(V, \tau_{BY})$ be a nuclear space as above, $K\subseteq V'$ and $L$ a linear functional on $S(V)$. $L$ has a $\overline{K}-$representing measure if and only if $L$ is $\tau_{BY}-$continuous and $L(\operatorname{Pos}(K))\subseteq[0, +\infty)$, where
$\operatorname{Pos}(K)$ is defined as in \eqref{PosK} for $A=S(V)$ and $\overline{K}$ denotes the weak closure of $K$ in $V'$.
\end{theorem}
In this setting, having in mind Question \ref{all-top} and Question \ref{weak-cont-nuclear}, it becomes natural to ask the following.
\begin{q}
Does Proposition \ref{PropC} or the results in Theorem~\ref{nuclear} hold for the class of lc nuclear spaces considered in \cite[Section 4]{Bor-Yng75}?
\end{q}
Note that the integral representation obtained is actually regarded in \cite{Bor-Yng75} as a decomposition
of a continuous $K-$positive functional into multiplicative functionals. This is indeed the more general point of view considered in this paper as well as in many others (see e.g. \cite{BLY}, \cite{Pow71-74}, \cite[Section 12.4]{Schmu90}) where interesting results about decompositions of linear functionals on a commutative (and even on an arbitrary) nuclear *- algebra into pure states were obtained. 

The results described in this section suggest that introducing more structure on the algebra considered in Problem \ref{GenKMP} can allow to get better results in those special cases and so open new directions towards interesting applications (e.g. considering the symmetric algebra $S(V)$ endowed with the topology introduced in Proposition \ref{PropC} allowed to explicit compute the support and so in some sense to get a better result than Proposition \ref{coroll-lmcAlg}). \begin{q}
Investigate Problem \ref{GenKMP} for other special class of algebras (e.g. algebra of polynomials invariant under the action of a group considered in \cite{CKS}, Stone algebras). 
\end{q}

Several questions remain open also in relation to the determinacy of the moment problem on the symmetric algebra of a locally convex space as described in \cite{IKM} (see also \cite{I} for a review about determinacy in both finite and infinite dimensions). Last but not least, one could obviously ask:
\begin{q}
What about the truncated case in this general setting? 
 \end{q}
There is a huge literature about the truncated moment problem in the finite dimensional case (see e.g. \cite{F-Survey}, \cite{Lau08}, \cite[Chap. III]{LasBook} for nice surveys about the results obtained so far) but there are still many unsolved questions and even less is known in the infinite dimensional case. As explained in the introduction, despite the infinite dimensional truncated moment problem is a longstanding problem in several applied fields,  there is a lack of a structural theoretical investigation of this problem. Some progress in this direction have been recently done e.g. in \cite{KLS07}, \cite{KLS09}, \cite{MolLach15}, where a systematic study of the truncated case for point processes and random sets was initiated. However, to the best of our knowledge, at the moment there is no results for the truncated moment problem for general $\RR-$algebras. In \cite{GIKM2} we are currently working together with Mehdi Ghasemi on a version of the achievements contained in Section 2 for the truncated case starting from some promising results and ideas discussed with Murray Marshall. 

\section*{Acknowledgments}
We would like to thank the anonymous referee for her or his helpful comments and suggestions. We also express our gratitude to Sabine Burgdorf, Mehdi Ghasemi and Victor Vinnikov for many interesting discussions.
\newpage
\normalsize
\bibliographystyle{amsalpha}

\end{document}